\documentclass[11pt]{amsart}
\usepackage{amsmath,amssymb,enumerate}
\usepackage{latexsym}
\usepackage{amsfonts}
\usepackage {amsbsy}
\usepackage {amsmath}
\usepackage{amssymb}
\usepackage{enumerate}

\newtheorem{theorem}{Theorem}
\newtheorem{lemma}[theorem]{Lemma}
\newtheorem{corollary}[theorem]{Corollary}

\newtheorem{proposition}[theorem]{Proposition}
\newtheorem{definition}[theorem]{Definition}
\newtheorem{example}[theorem]{Example}

\numberwithin{theorem}{section}
\numberwithin{equation}{section}

\newcommand{\B}{\mathbb{B}}
\newcommand{\R}{\mathbb{R}}
\newcommand{\C}{\mathbb{C}}
\newcommand{\Cc}{\mathcal{C}}

\newcommand{\al}{\alpha}
\newcommand{\ep}{\epsilon}
\newcommand{\Om}{\Omega}
\newcommand{\Omf}{\partial\Omega}
\newcommand{\Omb}{\bar{\Omega}}
\newcommand{\fami}{\mathcal{V}_m(\Om,\fii,f)}
\newcommand{\fii}{\varphi}
\newcommand{\D}{\Delta_H }
\newcommand{\U}{\mathtt{U}}

\begin{document}
\title[ complex Hessian equations]{Modulus of continuity of  solutions to  complex Hessian equations}
\author{Mohamad CHARABATI}

\date{\today}

\begin{abstract}
We give a sharp estimate of the modulus of continuity of the solution to the Dirichlet problem for the complex Hessian equation of order $m$ ($1 \leq m \leq n$) with a continuous right hand side  and a continuous boundary data in a bounded strongly $m$-pseudoconvex domain  $\Om \Subset \C^n$. Moreover when the right hand side is in $L^p(\Om) $, for some $p > n/m$ and the boundary value function  is $C^{1,1}$ we prove that the solution is H\"older continuous.
\end{abstract}

\maketitle

\section{introduction}

Let $\Om$ be a  bounded  domain in $\C^n$ with smooth boundary and $ m$ be an integer such that $ 1 \leq m \leq n$. Given $\fii\in\Cc(\Omf) $ and $ 0\leq  f\in\Cc(\Omb)  $.
We consider the Dirichlet problem for complex Hessian equation:

\begin{equation}
\label{hess}
\begin{cases}
	u \in SH_m(\Om) \cap \Cc(\Omb) \\
 	(dd^cu)^m\wedge \beta^{n-m}
	= f \beta^n & \text{ in } \;\; \Om \\
 	u= \fii & \text{ on } \;\; \Omf 
\end{cases}
\end{equation}
where $ SH_m(\Om)$ denote the set of all $ m$-subharmonic functions in $\Om$  and  $ \beta := dd^c |z|^2 $ is the standard K\"ahler form in $\C^n$.

\smallskip

In the  case $ m=1$, this equation corresponds to the Poisson equation which is classical. The case $m=n$  corresponds to the complex Monge-Amp\`ere equation which was intensively studied these last decades  by several authors (see \cite{BT76}, \cite{CP92}, \cite{CK94}, \cite{Bl96} \cite{Ko98}, \cite{GKZ08}).

The complex Hessian equation is a new subject and is much more difficult to handle than the complex Monge-Amp\`ere equation (e.g. the $m$-subharmonic functions are not invariant under holomorphic change of variables, for $ m < n$).
Despite that, the pluripotential theory which was developed in (\cite{BT82} , \cite{D89}, \cite{Ko98}) for the complex Monge-Amp\`ere equation can be adapted to the complex Hessian equation.

\smallskip

The Dirichlet problem (\ref{hess}) was considered by Li in \cite{Li04}. He proved that if $ \Om$ is a bounded strongly $m$-pseudoconvex domain with smooth boundary (see the definition below), $\fii \in \Cc^\infty(\Omf)$ and $ 0< f \in \Cc^\infty(\Omb) $ then there exists a unique smooth solution to the Dirichlet problem (\ref{hess}).

\noindent
The existence of continuous solution for the homogenous Dirichlet problem in the unit ball was proved by Z. Blocki \cite{Bl05}.

The H\"older continuity of the solution when the right hand side and the boundary data are H\"older continuous was proved by H.C. Lu \cite{Lu12}.

\noindent
Recently, S. Dinew and S. Kolodziej proved in \cite{DK11} that there exists a unique continuous  solution to (\ref{hess}) when $ f \in L^p (\Om)$ , $p > n/m$. The H\"older continuity of the solution in this case has been studied independently by H.C. Lu \cite{Lu12} and Nguyen \cite{N13}.

\noindent 
A viscosity approach to the complex Hessian equation has been developed by H.C. Lu in  \cite{Lu13b}.

\noindent
A potential theory for the complex Hessian equation was developed by Sadullaev and Abdullaev in \cite{SA12} and H.C.Lu \cite{Lu12} at the same time.

\smallskip

Our first main result in this paper gives a sharp estimate for the modulus of continuity of the solution to Dirichlet problem for  complex Hessian equation (\ref{hess}).

More precisely, we will prove the following result.

\vskip.2cm

\begin{theorem}\label{main}
Let $ \Om   $ be a smoothly bounded strongly $m$-pseudoconvex domain in $\C^n$, assume  that $ 0 \leq f \in \Cc(\Omb)$ and $ \fii \in \Cc(\Omf)$. Then the modulus of continuity $\omega_\U$ of the solution $\U$  satisfies  the following estimate
   $$ \omega_\U (t) \leq   \tau (1+ \|f\|^{1/m}_{L^\infty (\Omb)}) \;  \max \{ \omega_\fii(t^{1/2}), \omega_{f^{1/m}}(t), t^{1/2} \}  $$
where $ \tau \geq 1 $ is a constant  depending only on $ \Om$.

\end{theorem}

In the case of the complex Monge-Amp\`ere equation, Y. Wang gave a control on the modulus of continuity of the solution assuming the existence of a subsolution and a supersolution with the given boundary data (\cite{W12}). 

Here we do not assume the existence of a subsolution and a supersolution. Actually the main argument in our  proof consists in constructing adequate barriers for the Dirichlet problem for the complex Hessian equation (\ref{hess}).

\bigskip
For the case when the density  $f \in L^p(\Om)$ with $ p > n/m$, C.H. Nguyen \cite{N13} proved the H\"older continuity of the solution when the boundary data is in $\Cc^{1,1}(\Omf)$ and the density $f$ satisfies a growth condition near the boundary of $\Om$.

For the case $m=n$, we proved recently (\cite{Ch14}) that the solution to the Dirichlet problem (\ref{hess}) is H\"older continuous on $\Omb$ without any condition near the boundary.
Using the same idea we can prove a similar result for the hessian equation. More precisely, we have the following theorem.  

\begin{theorem}\label{main2}
Let $\Om \subset \C^n$ be a  bounded  strongly $m$-pseudoconvex  domain with smooth boundary,   $ \fii \in \Cc^{1,1} (\Omf) $ and $ 0 \leq f \in L^p(\Om)$, for some $ p > n/m $. Then the solution $\U$ to (\ref{hess}) belongs to $  \Cc^{0,\al}(\Omb)$ for any $ 0<\al < \gamma_1 $. Moreover, if $p \geq 2 n/m$ then the solution to the Dirichlet problem  $ \U \in  \Cc^{0,\al}(\Omb)$ for any $ 0<\al <  \min \{ \frac{1}{2}, 2 \gamma_1 \} $,  where $ \gamma_r:=   \frac{r}{r+mq+\frac{pq(n-m)}{p- \frac{n}{m}}} $ and $r \geq 1$.
\end{theorem}

\vskip.2cm

\section{Preliminaries}

In this section, we briefly recall some facts from linear algebra and basic results from potential theory for $m$-subharmonic functions. We refer the reader to \cite{Bl05, DK11, Lu13a, Lu13c, SA12} for more details.

\noindent
Let us set
$$ H_m (\lambda) = \sum\limits_{ 1 \leq j_1 < ...< j_m \leq n} \lambda_{j_1} ...\lambda_{j_m}, $$ 
where $ \lambda = (\lambda_1, ..., \lambda_n) \in \R^n$.\\
Thus $ (t+\lambda_1) ...(t+\lambda_n) = \sum\limits_{m=0}^{n} H_m(\lambda) t^{n-m} $ for $ t \in \R$, where $ H_0(\lambda)=1$.

\noindent 
We denote $\Gamma_m$ the closure of the connected component of $\{ H_m>0\}$ containing $ (1,1,...,1)$. One can show that 
$$ \Gamma_m = \{ \lambda \in \R^n : H_m(\lambda_1+t, ..., \lambda_n+t) \geq 0 , \; \forall t \geq 0\}. $$
It follows from the identity 
$$ H_m(\lambda_1+t, ..., \lambda_n+t) = \sum\limits_{p=0}^{m} \binom{n-p}{m-p} H_p (\lambda) t^{m-p} ,$$ 
that 
$$ \Gamma_m = \{ \lambda \in \R^n : H_j(\lambda) \geq 0, \; \forall 1 \leq j \leq m \}.$$
It is clear that $ \Gamma_n \subset \Gamma_{n-1} \subset ... \subset \Gamma_1$, where $  \Gamma_n = \{ \lambda \in \R^n : \lambda_i \geq 0 \; \forall i \} $.

By the paper of G{\aa}rding \cite{G59}, the set $ \Gamma_m$ is a convex cone in $\R^n$ and $ H_m^{1/m} $ concave on $\Gamma_m$. By Maclaurin inequality, we get 
$$ \binom {n}{m}^{-1/m} H_m^{1/m} \leq \binom{n}{p}^{-1/p} H_p^{1/p} ; \;\; 1 \leq p \leq m \leq n .$$

Let  $\mathcal{H}$ denote the vector space over $\R$ of complex hermitian $n\times n $ matrices. For any $A \in \mathcal{H}$, let $ \lambda (A) = (\lambda_1,..., \lambda_n) \in \R^n$ be the eigenvalues of $A$. We set
$$ \tilde{H}_m(A)= H_m(\lambda(A)).$$
Now, we define the cone 
$$ \tilde{\Gamma}_m := \{ A\in \mathcal{H} : \lambda(A) \in \Gamma_m \} =\{ A \in \mathcal{H} : \tilde{H}_j (A) \geq 0 , \; \forall 1\leq j \leq m \}.$$

Let $\alpha$ be a real (1,1)-form determined by 
$$ \alpha = \frac{i}{2}\sum\limits_{i,j} a_{i\bar{j}} dz_i \wedge d\bar{z}_j $$
where $(a_{i\bar{j}})$ is a hermitian matrix. After diagonalizing the matrix $ A= (a_{i\bar{j}})$, we see that 
$$ \alpha^m \wedge \beta^{n-m} = \tilde{S}_m(\alpha) \beta^n ,$$
where $\beta$ is the standard K\"ahler form in $\C^n$ and $ \tilde{S}_m(\alpha)= m! (n-m)!  \tilde{H}_m (A) $.

\noindent
The last equality allows us to define 
$$ \hat{\Gamma}_m := \{ \alpha \in \C_{(1,1)} : \alpha \wedge \beta^{n-1} \geq 0, \alpha^2 \wedge \beta^{n-2} \geq 0, ..., \alpha^m \wedge \beta^{n-m} \geq 0 \} .$$
where $\C_{(1,1)}$ is the space of real (1,1)-forms with constant coefficients.

\vskip.2cm
Let $M: \C_{(1,1)}^m \to \R $ be the polarized form of $\tilde{S}_m$, i.e. $M$ is linear in every variable, symmetric and $ M(\al,...,\al)= \tilde{S}_m(\al), \; \text{ for any } \al \in \C_{(1,1)}.$
\\
The G{\aa}rding  inequality  (see \cite{G59}) asserts that
\begin{equation}\label{gar}
M(\al_1, \al_2,...,\al_m) \geq \tilde{S}_m(\al_1)^{1/m} ...  \tilde{S}_m(\al_m)^{1/m} , \;\; \al_1,\al_2,...,\al_m \in \hat{\Gamma}_m .
\end{equation}

\begin{proposition}(\cite{Bl05}).
If $\alpha_1 , ..., \alpha_p \in \hat{\Gamma}_m$, $1\leq p \leq m$, then we have
$$\alpha_1 \wedge \alpha_2 \wedge ...\wedge \alpha_p \wedge \beta^{n-m} \geq 0 . $$ 
\end{proposition}

\noindent
Let us set 
$$ \Sigma_m:= \{ \alpha \in \hat{\Gamma}_m   \text{ of constant coefficients such that } \tilde{S}_m(\alpha)=1\}.$$
Recall the following elementary lemma and we include its proof for the convenience of the reader.

\begin{lemma}\label{form}
Let $\al \in \hat{\Gamma}_m$. Then the following identity holds 
$$ \tilde{S}_m(\al)^{1/m} = \inf \left\{ \frac{\al \wedge \al_1 \wedge ... \wedge \al_{m-1} \wedge \beta^{n-m} }{\beta^n} ; \al_i \in \Sigma_m , \forall i \right\}.$$
\end{lemma}

\begin{proof}
Let $M$ be a polarized form of $\tilde{S}_m $ define by 
$$M(\al, \al_1, ..., \al_{m-1}) = \frac{\al \wedge \al_1 \wedge ... \wedge \al_{m-1} \wedge \beta^{n-m}}{\beta^n} ,  $$
for $\al_1,...,\al_{m-1} \in \Sigma_m , \al \in \hat{\Gamma}$. By Garding inequality \ref{gar}, we have
$$ M(\al, \al_1, ..., \al_{m-1}) \geq \tilde{S}_m(\al)^{1/m}. $$
Then we obtain that 
$$ \tilde{S}_m(\al)^{1/m} \leq \inf \left\{ \frac{\al \wedge \al_1 \wedge ... \wedge \al_{m-1} \wedge \beta^{n-m} }{\beta^n} ; \al_i \in \Sigma_m , \forall i \right\}.$$
Now, setting $ \al_1 = ...= \al_{m-1} = \frac{\al}{\tilde{S}_m(\al)^{1/m}}$, we can ensure that
$$ M(\al, \al_1, ..., \al_{m-1}) = \tilde{S}_m(\al)^{1/m} .$$
This complete the proof of lemma.
\end{proof}

\bigskip
\it Aspects about  $m$-subharmonic functions.
\rm Let $\Om \subset \C^n$ is a bounded domain. Let also $ \beta:=dd^c|z|^2 $ is the standard K\"ahler form in $\C^n$.

\begin{definition}(\cite{Bl05}).
Let $u$ be a subharmonic function in $\Om$.\\
1) For smooth case, $u \in \Cc^2(\Om)$ is said to be $m$-subharmonic (briefly $m$-sh) if the form $dd^c u$ belongs pointwise to $\hat{\Gamma}_m$.\\
2) For non-smooth case, $u$ is called $m$-sh if for any collection $\al_1,\al_2,...,\al_{m-1} \in \hat{\Gamma}_m$ , the inequality 
$$ dd^c u \wedge \al_1 \wedge ...\wedge \al_{m-1} \wedge \beta^{n-m} \geq 0$$
holds in the weak sense of currents in $\Om$.
\end{definition}

We denote $SH_m(\Om)$ the set of all $m$-sh functions. We will recall some properties of $m$-sh functions.
\begin{proposition}(\cite{Bl05}).

\begin{enumerate}

\item $ PSH = SH_n  \subset  SH_{n-1} \subset ... \subset  SH_1 =SH.$
\item If $u,v \in SH_m(\Om) $ then $ \lambda u + \eta v \in SH_m(\Om)$, $ \forall \lambda,\eta \geq 0$
\item If $u \in SH_m(\Om)$ and $\gamma : \R \to \R $ is convex increasing function then $\gamma \circ u \in SH_m(\Om)$. 
\item If $u \in SH_m(\Om)$ then the standard regularizations $u_\ep = u*\rho_\ep$ are also $m$-subharmonic in $\Om_\epsilon := \{ z\in \Om | dist (z, \partial \Om) > \epsilon\}$.
\item Let $U$ ba a non empty proper open subset of $\Om$, if $u \in SH_m(\Om)$, $v \in SH_m(U) $ and $ \lim\limits_{z \to y} \sup v(z) \leq u(y) $ for every $ y \in \partial U \cap \Om$, then the function 
\begin{center}

$ w=\left\{ \begin{array}{ll}
                   max(u,v) & \text{in} \; U \\
                    u &  \text{in} \; \Om \setminus U  \\
                   \end{array} \right. $    
\end{center} 
is $m$-sh in $\Om$.
\item Let $\{u_\alpha\} \subset SH_m(\Om)$ be locally uniformly bounded from above and $u= \sup u_\alpha$. Then the upper semi-continuous regularization $u^*$ is $m$-sh and equal to $u$ almost everywhere.
\end{enumerate}
\end{proposition}

\medskip

For locally bounded $m$-subharmonic functions, we can inductively define a closed nonnegative current (following Bedford and Taylor for plurisubharmonic functions).
$$ dd^c u_1 \wedge ... \wedge dd^c u_p \wedge \beta^{n-m} := dd^c(u_1 dd^c u_2 \wedge ...\wedge dd^c u_p \wedge \beta^{n-m} ), $$ 
where $u_1,u_2,...,u_p \in SH_m(\Om)\cap L_{loc}^{\infty}(\Om)$, $p\leq m$.\\
In particular, we define the nonnegative Hessian measure of $ u \in SH_m(\Om) \cap L_{loc}^{\infty}(\Om)$ to be 
$$ H_m(u):= (dd^cu)^m \wedge \beta^{n-m}. $$

Let us define the differential operator
 $ L_\al : SH_m(\Om) \cap L^{\infty}_{loc} (\Om)  \to \mathcal{D}^\prime (\Om) $
such that 
$$ dd^c u \wedge \al_1 \wedge ... \wedge \al_{m-1} \wedge \beta^{n-m} = L_\al u  \beta^n ,$$
where $ \al_1, ..., \al_{m-1} \in \Sigma_m $.
In appropriate complex coordinates this  operator is the Laplace operator. 

\begin{example}

Using Garding inequality \ref{gar}, one can note that $ L_\al (|z|^2) \geq 1 $ for any $\al_i \in \Sigma_m , 1 \leq i \leq m-1$.

\end{example}

We will prove the following essential proposition by applying ideas from the viscosity theory developed in  
\cite{EGZ11} for the complex Monge-Amp\`ere equation and extended to the complex Hessian equation by H.C.Lu (\cite{Lu12}, \cite{Lu13b}).
\begin{proposition}\label{equiv}
Let $ u \in SH_m(\Om) \cap L^{\infty}_{loc}(\Om) $ and $ 0 \leq f \in \Cc(\Omb) $. The following conditions are equivalent:\\
1) $ L_\al u \geq f^{1/m} , \forall \al_1 , ..., \al_{m-1} \in \Sigma_m$.\\
2) $ (dd^c u)^m \wedge \beta^{n-m} \geq f \beta^n $ in $ \Om$.

\end{proposition}

\begin{proof}
First observe that if $ u \in \Cc^2(\Om) $, then by Lemma \ref{form} we can see that (1) is equivalent to
 $$ \tilde{S}_m(\al)^{1/m}  \geq f^{1/m}, $$
 where $ \al = dd^c u $ is a real (1,1)-form belongs to $ \hat{\Gamma}_m $.
 \\
The last inequality corresponds to 
 $$ (dd^c u)^m \wedge \beta^{n-m} \geq f \beta^n \;  \text{in} \; \Om .$$ 
 (1  $ \Rightarrow $ 2)
We consider the standard regularization $u_\epsilon$ of $u$ by convolution with a smoothing kernel.
We then get
$$ L_\al u_\epsilon \geq  (f^{1/m})_\epsilon. $$
Since $u_\epsilon$ is smooth, by the  observation above, we have 
$$ (dd^c u_\epsilon )^m \wedge \beta^{n-m} \geq ((f^{1/m})_\epsilon )^m \beta^n.$$
Letting $ \epsilon \to 0$, by the convergence theorem for the Hessian operator under decreasing sequence, we get
 $$ (dd^c u)^m \wedge \beta^{n-m} \geq f \beta^n  \; \text{in} \; \Om .$$ 
(2  $ \Rightarrow $ 1)
Fix $ x_0 \in \Om $ and $q$ is $\Cc^2$-function in a neighborhood $V \Subset \Omega $ of $x_0$ such that  $u \leq q $ in this neighborhood and $u(x_0)= q(x_0)$.
We will prove that 
$$ (dd^c q)^m_{x_0} \wedge \beta^{n-m} \geq f(x_0) \beta^{n} .$$
First step: 
we claim that $ dd^c q_{x_0} \in \hat{\Gamma}_m$.
\\
If $u$ is smooth, we note that $x_0$ is a local minimum point of $q -u$, then $ dd^c (q-u)_{x_0} \geq 0$. 
Hence we see that $ (dd^c q)^k \wedge \beta^{n-k}  \geq 0 $ in $x_0$, for $ 1 \leq k \leq m$.
This gives that 
$ dd^c q_{x_0} \in \hat{\Gamma}_m$.
\\
If $u$ non smooth, let $u_\epsilon$ is the standard smooth regularization of $u$, then $u_\epsilon$ is $m$-sh, smooth and $  u_\epsilon \searrow u$.
Now let us fix $\delta >0$ and $\epsilon_0 >0$ such that the neighborhood of $x_0$, $ V \subset \Omega_{\epsilon_0} $. For each $\epsilon < \epsilon_0$, let $y_\epsilon$ be the maximum point of $ u_\epsilon - q - \delta |x-x_0|^2 $ on $\bar{B} \Subset V$ (where $B$ is a small ball centered at $x_0$).
Then we have
$$ u_\epsilon(x) -q(x) -\delta |x-x_0|^2 \leq u_\epsilon (y_\epsilon) - q(y_\epsilon) - \delta |y_\epsilon - x_0|^2 .$$  
Assume that $ y_\epsilon \to y \in \bar{B}$. Let us put $x=x_0$ an passing to the limit in the last inequality, we obtain
$$ 0 \leq u(y) - q(y) - \delta |y-x_0|^2 ,$$
but $ q \geq u $ in $ V$, then we can conclude that $ y=x_0$.
\\
Let us then define 
$$ \tilde{q}:= q +\delta |x-x_0|^2 + u_\epsilon(y_\epsilon) -q(y_\epsilon) - \delta |y_\epsilon - x_0|^2.$$
which is $\Cc^2$-function in $B$ and satisfies $ u_\epsilon(y_\epsilon) = \tilde{q}(y_\epsilon)$ and $ \tilde{q} \geq u_\epsilon$ in $B$, then the following inequality  holds in $y_\epsilon$
$$ (dd^c \tilde{q})^k \wedge \beta^{n-k} \geq 0  \text{ for } 1 \leq k \leq m, $$
that is 
$$ (dd^c q +\delta \beta)^k_{y_\epsilon} \wedge \beta^{n-k} \geq 0 \text{ for } 1 \leq k \leq m. $$
Letting $ \epsilon $ tend to 0,  we get 
$$ (dd^c q +\delta \beta)^k_{x_0} \wedge \beta^{n-k} \geq 0  \text{ for } 1 \leq k \leq m .  $$
Since the last inequality holds for any $ \delta >0$, we can get that
$ dd^c q_{x_0} \in \hat{\Gamma}_m$.

Second step:  
assume that there exist a point 
$ x_0 \in \Omega $ and a $\Cc^2 $-function $q$ satisfy $ u \leq q $ in a neighborhood of $x_0$ and 
$ u(x_0)=q(x_0) $ such that  
$$ (dd^c q)_{x_0}^m \wedge \beta^{n-m}  < f(x_0) \beta^n .$$

\noindent
Let us put
$$ q^\ep (x) = q(x) + \ep ( \| x-x_0 \|^2  - \frac{r^2}{2}) $$
which is $\Cc^2$- function and for  $ 0< \ep \ll 1 $ small enough we have 
 $$ 0< (dd^c q^\ep )_{x_0}^m \wedge \beta^{n-m} < f(x_0) \beta^n . $$
Since $f$ is continuous on $\Omb$, there exists  $ r>0 $ such that 
\begin{center}
  $ (dd^c q^\ep )^m \wedge \beta^{n-m} \leq f \beta^n   \;\;  \text{in}\; \B(x_0,r) $.
\end{center}
Hence, we get
$$ (dd^c q^\ep )^m \wedge \beta^{n-m} \leq (dd^c u )^m \wedge \beta^{n-m}   \;\;  \text{in}\; \B(x_0,r) $$
and $ q^\epsilon = q + \epsilon r^2/2 > q \geq u $ on $ \partial \B(x_0,r)$. From the comparison principle ( see \cite{Bl05, Lu12}), it follows that $ q^\epsilon \geq u $ in $ \B(x_0,r)$, but  this contradicts that $ q^\epsilon (x_0)=u(x_0)- \epsilon r^2 /2 < u(x_0)$.

We have shown that  for every point  $ x_0 \in \Om $, and every $ \Cc^2 $-function  $ q $ in a neighborhood of $ x_0$ such that $ u \leq q $ in this neighborhood  and $ u(x_0) = q(x_0) $, we have
$ (dd^c q)^m_{x_0} \wedge \beta^{n-m} \geq f(x_0) \beta^{n} $, hence we have $L_\al q_{x_0} \geq f^{1/m}(x_0) $.

\smallskip
Final step to prove 1). Assume that $f>0 $ is smooth function. Then there exists a $ \Cc^\infty $-function, say $g$ such that $ L_\al g = f^{1/m} $ . Hence Theorem 3.2.10' in \cite{H94} implies that $ \fii = u-g$ is $ L_\al$-subharmonic, consequently $L_\al u \geq f^{1/m}$.
\\
In the case $f>0$ is only continuous. Note that 
$$ f = \sup \{ w \in \Cc^\infty (\Omb) , 0<w\leq f \}. $$
Since $ (dd^c u )^m \wedge \beta^{n-m} \geq f \beta^n $, we get $ (dd^c u )^m \wedge \beta^{n-m} \geq w \beta^n $. As $w>0 $ smooth , we can see that $ L_\al u \geq w^{1/m}$ , therefore $ L_\al u \geq f^{1/m} $.
\\
In general case, $ 0 \leq f \in \Cc(\Omb)$. We observe that $ u_\epsilon (z) =u(z) + \epsilon |z|^2 $ satisfies
 $$ (dd^c  u_\epsilon)^m \wedge \beta^{n-m} \geq (f+ \epsilon^m) \beta^n .$$
By the last step, we get $ L_\al u_\epsilon \geq (f+\epsilon^m)^{1/m} $ , then the wanted result follows by letting $\epsilon$ tend to $0$.
\end{proof}

\begin{definition}\label{def}
Let $\Om \subset \C^n$ be a smoothly bounded domain, we say that $\Om$ is strongly $m$-pseudoconvex if there exist a defining function $\rho$ of $\Om$ (i.e. a smooth function in a neighborhood $U$ of $\Omb$ such that $ \rho <0$ on $\Om$, $\rho=0 $ and $d \rho \neq 0 $ on $ \Omf$ ) and $A>0$ such that 
$$ (dd^c \rho )^k \wedge \beta^{n-k} \geq A \beta^n \;  \text{in} \; U, \; \; \text{for} \; 1 \leq k \leq m .$$ 

\end{definition}

The existence of a solution $\U$ to Dirichlet problem \ref{hess} was proved in \cite{DK11}. This solution can be  given by the upper envelope of subsolutions to the Dirichlet problem as in \cite{BT76} for the complex Monge-Amp\`ere equation.
\begin{equation}\label{sup1}
 \U= \sup \{ v \in SH_m(\Om) \cap \Cc(\Omb) ; \; v|_{\Omf} \leq \fii \; \text{and}\; (dd^c v)^m \wedge \beta^{n-m} \geq f \beta^n \}.
\end{equation} 

However, thanks to Lemma \ref{equiv}, we can describe the solution as the following 
\begin{equation}\label{sup2}
\U = \sup \{ v \in \fami \},
\end{equation}
where the nonempty  familly $ \fami$ is defined as 
$$ \mathcal{V}_m = \{ v \in SH_m(\Om) \cap \Cc(\Omb) ; \; v|_{\Omf} \leq \fii \; \text{and}\; L_\al v \geq f^{1/m} , \forall \al_i \in \Sigma_m , 1 \leq i \leq m-1  \}.$$ 
\\
Observe that the description of the solution in formula (\ref{sup2}) is more convenient, since it deals with subsolutions with respect to a family of linear elliptic operators.
\bigskip

\section{The modulus of continuity of  the solution }

Recall that a real function $\omega$ on $[0,l]$, $0<l<\infty$, is called a modulus of continuity if $\omega$ is continuous, subadditive, nondecreasing and $\omega(0)=0$.\\
In general, $\omega$ fails to be concave, we denote $\bar{\omega}$ to be the minimal concave majorant of $\omega$. For more details, we refer the reader to \cite{T63, Kor82}.
We denote $\omega_{\psi}$ the optimal modulus of continuity of the continuous function $ \psi$ which is defined by 
$$ \omega_\psi (t)= \sup\limits_{|x-y|\leq t} |\psi(x)-\psi(y)|. $$
The following lemma is one of the useful properties of $\bar{\omega}$.

\begin{lemma}\label{mod}
Let $\omega$ be a modulus of continuity on $[0,l]$ and  $\bar{\omega}$  be the minimal concave majorant of $\omega$. Then $ \omega(\eta t)< \bar{\omega}(\eta t) < (1+\eta) \omega(t)$ for any $t>0$ and $\eta >0$ .

\end{lemma}

\begin{proof}
Fix $t_0>0 $ such that $ \omega(t_0) > 0$. We claim that 
$$ \frac{\omega(t)}{\omega(t_0)} \leq 1+ \frac{t}{t_0}, \;\; \forall t\geq 0.$$
For $ 0 < t \leq t_0$, since $\omega$ is nondecreasing, we have 
$$ \frac{\omega(t)}{\omega(t_0)} \leq \frac{\omega(t_0)}{\omega(t_0)} \leq 1+ \frac{t}{t_0}$$
Otherwise , if $t_0 \leq t \leq l$, by the Euclid's Algorithm, we write $ t= k t_0 + \alpha $, $ 0 \leq \alpha < t_0$  and $k$ is natural number with $ 1 \leq k \leq t/t_0$ .
Using the subadditivity of $\omega$, we observe that 
$$ \frac{\omega(t)}{\omega(t_0)} \leq \frac{k\omega(t_0)+ \omega(\alpha)}{\omega(t_0)} \leq k+1 \leq 1 + \frac{t}{t_0}.$$
Let $ l(t):=   \omega(t_0) + \frac{t}{t_0} \omega(t_0)$ is a straight line, then $ \omega (t) \leq l(t)$ for all $ 0 < t \leq l $.\\
Therefore
$$ \bar{\omega} (t) \leq l(t)= \omega(t_0) + \frac{t}{t_0} \omega(t_0)$$
 for all $ 0 < t \leq l$. Hence, for any $ \eta > 0 $ we have
$$ \omega(\eta t)< \bar{\omega}(\eta t) < (1+\eta) \omega(t).$$

\end{proof}

In the following proposition, we establish a barrier to the problem \ref{hess} and estimate its modulus of continuity.

\begin{proposition}\label{subs}
Let $ \Om \subset \C^n $ be a bounded strongly $m$-pseudoconvex  domain with smooth boundary, assume  that $ \omega_\fii $ is the modulus of continuity of $\fii \in \Cc(\Omf )$ and $ 0 \leq f \in \Cc(\Omb) $. Then there exists a subsolution $ v \in \fami$ such that $ v = \fii $ on $ \Omf$ and  the modulus of continuity of $v$ satisfies the following inequality  
  $$ \omega_v(t) \leq  \lambda \max \{ \omega_\fii(t^{1/2}), t^{1/2} \},   $$
where  $\lambda = \eta (1+ \|f \|^{1/m}_{L^\infty(\Omb)})$ and $\eta \geq 1 $ is a  constant depending on $ \Om $.

\end{proposition}

\begin{proof}
Fix  $ \xi  \in  \Omf  $ , we will prove that there exists  $ v_\xi \in \fami  $  such that $ v_\xi(\xi)=\fii(\xi) $. \\ 
We claim that there exists  a constant $ C>0$, depending only on $\Om $, such that for every point $ \xi \in \Omf $  and $ \fii \in \Cc(\Omf) $, there is a function $ h_{\xi} \in SH_m(\Om) \cap \Cc(\Omb) $ such that\\
1) $ h_\xi(z) \leq \fii(z)  , \forall z \in \Omf $ \\
2) $h_\xi(\xi)= \fii(\xi) $\\
3) $ \omega_{h_\xi } (t) \leq C \omega_{\fii} (t^{1/2})$.\\
Assume that the claim is proved. Fix a point $z_0 \in \Om$ and choose $K_1 \geq 0$ such that $K_1 = \sup_{\Omb} f^{1/m}  $, hence $ L_\al(K_1 | z -z_0 |^2 ) = K_1 L_\al| z -z_0 |^2 \geq f^{1/m}(z) $ for all $ \al_i \in \Sigma_m, 1 \leq i \leq m-1 $ and let us set  $ K_2=K_1 | \xi -z_0  |^2$. Then for the continuous function
$$ \tilde{\fii}(z):= \fii(z) - K_1 | z -z_0 |^2 + K_2 . $$
we have $ h=h_\xi$ such that   1),2)and  3) hold.
\\
Then the desired function   $ v_\xi \in \fami  $ is given by  
$$ v_\xi(z) := h(z) + K_1 | z -z_0 |^2 - K_2 .$$ 
Indeed, $ v_\xi \in SH_m(\Om) \cap \Cc(\Omb)  $ and we have 
$$ h(z) \leq \tilde{\fii}(z)= \fii(z) -K_1 | z -z_0 |^2 + K_2  \text{  on } \Omf. $$
So  that $ v_\xi(z) \leq \fii(z) $ on  $ \Omf$ and  $ v_\xi(\xi)= \fii(\xi) $.
 It is clear that  
$$ L_\al v_\xi = L_\al h + K_1 L_\al | z -z_0 |^2  \geq f^{1/m} \text{   in } \Om $$ 
Moreover, by the hypothesis , we can get an estimate for the modulus of continuity of $v_\xi$
 \begin{center}
 $\begin{array}{ll}
 \omega_{v_\xi} (t)= \sup \limits_{|z-y| \leq t } |v(z)-v(y)|  & \leq \omega_h (t) + K_1 \omega_{|z-z_0|^2} (t)\\
   & \leq C \omega_{\tilde{\fii}}(t^{1/2}) + 4 d^{3/2} K_1 t^{1/2} \\
  & \leq  C \omega_\fii (t^{1/2}) +2d K_1(C+2d^{1/2}) t^{1/2} \\
  & \leq (C+2d^{1/2})(1+2dK_1) \max \{ \omega_\fii (t^{1/2}) , t^{1/2} \}. \\
\end{array} $

 \end{center}
 Then we have
  $$  \omega_{v_\xi} (t) \leq \eta (1+  K_1 ) \max \{ \omega_\fii (t^{1/2} ) , t^{1/2} \}, $$
where $\eta:= (C+2d^{1/2})(1+2d) $ is a  constant depending on $ \Om$ and $ d := diam(\Om)$. 
\\
Now, we prove the claim.
Since $ \rho$ is smooth, we can choose $B>0$ large enough such that the function 
$$ g(z)=  B \rho (z) - | z-\xi |^2 $$
is $m$-subharmonic on $\Om$.
Let us set 
$$ \chi (t) = - \bar{\omega}_\fii ((-t)^{1/2}), $$
 for $ t \leq 0$ which is convex nondecreasing function on $ [-d^2 , 0] $. Now, fix $ r>0 $ so small that $ |g(z)|\leq d^2$ in $ B(\xi,r) \cap \Om$ and define for $ z \in B(\xi,r) \cap \Omb $ the function
$$ h(z) = \chi \circ g(z) + \fii(\xi) .$$
It is clear that $h$ is continuous $m$-subharmonic function on $ B(\xi,r) \cap \Om $ and one can observe that $ h(z) \leq \fii (z) $ if $ z \in B(\xi ,r) \cap \Omf $ and $ h(\xi ) = \fii (\xi) $.
Moreover, by the subadditivity of $ \bar{\omega}_\fii$ and Lemma \ref{mod} we have
\begin{center}
$  \begin{array}{ll}
       \omega_h(t)& = \sup\limits_{ | z-y| \leq t } | h(z) - h(y) | \\
        & \leq \sup\limits_{ | z-y| \leq t } \bar{\omega}_{\fii} \left[ \left| | z- \xi |^2 - | 
        y- \xi|^2 -B(\rho(z) - \rho(y) ) \right|^{1/2} \right]
        \\
       & \leq  \sup\limits_{ | z-y| \leq t } \bar{\omega}_{\fii} \left[ \left((2d+ B_1 ) |z-y|  
       \right)^{1/2} \right] \\
       & \leq  \tilde{C} . \omega_\fii (t^{1/2}) , \\
  \end{array}
$
\end{center}
where $\tilde{C}:= 1+(2d+B_1)^{1/2}$ is a constant depending on $\Om$.
\\
Recall that $ \xi \in \Omf$ and fix $ 0 < r_1 < r$ and $ \gamma_1 \geq d/r_1 $ such that 
$$ - \gamma_1 \bar{\omega}_\fii \left[ ( | z- \xi |^2 - B \rho (z) )^{1/2} \right] \leq \inf\limits_{\Omf} \fii - \sup\limits_{\Omf} \fii, $$
for $ z \in \Omf \cap \partial B(\xi , r_1) $. Let us set  $ \gamma_2 = \inf\limits_{\Omf} \fii  $,  then it follows that
 $$ \gamma_1 (h(z)-\fii(\xi))+\fii(\xi)  \leq \gamma_2 \text{ for } z \in \partial B(\xi, r_1) \cap \Omb .$$
Let us put 
 $$ h_\xi(z)=\left\{ \begin{array}{ll}
                   max[\gamma_1 (h(z)-\fii(\xi))+\fii(\xi) , \gamma_2 ] & ;z \in \Omb \cap ( 
                   B(\xi,r_1)  \\
                    \gamma_2 & ; z  \in \Omb \setminus B(\xi,r_1) . \\
                   \end{array} \right. $$  
This is a well defined $m$-subharmonic function on $\Om$ and continuous on $\Omb$. Moreover, it satisfies
 $ h_\xi(z) \leq \fii(z) $ for all $ z \in \Omf$. Since on $ \Omf \cap B(\xi,r_1)$, we have
$$ \gamma_1 (h(z) - \fii(\xi) ) + \fii(\xi) = - \gamma_1 \bar{\omega}_\fii ( | z- \xi | ) + \fii(\xi) \leq - \bar{\omega}_\fii ( | z- \xi | ) + \fii(\xi) \leq \fii(z). $$ 
Furthermore, the modulus of continuity of $h_\xi$ satisfies 
$$ \omega_{h_\xi}(t) \leq C \omega_\fii (t^{1/2}),$$
where $ C:= \gamma_1 \tilde{C}$ depends on $\Om$.
Hence it is clear that $ h_\xi$ satisfies the three conditions above.\\   
We have just proved that for each $\xi \in \Omf $, there is a function 
$$ v_\xi \in \fami, \; v_\xi(\xi)=\fii(\xi)  , \;  and \;\; \omega_{ v_\xi} (t) \leq \lambda \omega_\fii (t^{1/2}) , $$
where $ \lambda:= \eta(1+K_1)$.
Let us set 
  $$ v(z) = sup \left\{ v_\xi(z) ; \xi \in \Omf \right\}. $$
We have $ 0 \leq \omega_v (t) \leq \lambda  \omega_\fii (t^{1/2} )  $, then $ \omega_v(t) $ converges to zero when $t$ converges to zero. Consequently, we get $v \in \Cc(\Omb) $ and $ v=v^* \in SH_m(\Om) $.  Thanks to Choquet lemma, we can choose a nondecreasing sequence $(v_j)$, where $ v_j \in \fami$, converging to $ v$ almost everywhere, so that 
$$ L_\al v = \lim\limits_{j \to \infty} L_\al v_j \geq f^{1/m} , \forall \al_i \in \Sigma_m. $$
 It is clear that   $ v(\xi) = \fii(\xi), \forall \xi \in \Omf $. Finally, we get 
  $ v \in \fami $,  $ v = \fii $ on  $\Omf$ and $ \omega_v (t) \leq \lambda \omega_\fii ( t^{1/2} ). $

\end{proof}

\begin{corollary}\label{sups}
Under the same assumption of Proposition \ref{subs}. There exists a $m$- superharmonic function  $ \tilde{v} \in   \Cc(\Omb) $ such that $ \tilde{v} = \fii $ on $ \Omf $ and 
$$\omega_{\tilde{v}} (t) \leq \lambda \max \{ \omega_\fii (t^{1/2} ), t^{1/2} \} , $$
where $ \lambda >0$ as in Proposition \ref{subs}.
\end{corollary} 

\begin{proof}
We can do the same construction as in the proof of Proposition \ref{subs} for the function $ \fii_1=- \fii \in \Cc(\Omf) $, then  we get $  v_1 \in \mathcal{V}_m(\Om,\fii_1,f)$  such that $ v_1 = \fii_1 $ on  $\Omf$ and $\omega_{v_1} (t) \leq \lambda  \max \{ \omega_\fii (t^{1/2}), t^{1/2} \}  $. Hence, we set 
$ \tilde{v} = - v_1 $ which is a $m$-superharmonic function on $\Om$, continuous on $\Omb$ and satisfies $ \tilde{v} = \fii $ on $ \Omf $ and $\omega_{\tilde{v}} (t) \leq \lambda  \max \{ \omega_\fii (t^{1/2}), t^{1/2} \} $.
\end{proof} 

\noindent 
\begin{proof}[\bf{ Proof of Theorem \ref{main}}]
Thanks to Proposition \ref{subs}, we obtain a subsolution    $ v \in \fami  $, $ v = \fii $ on $\Omf$ and $ \omega_v (t) \leq \lambda \max \{ \omega_\fii (t^{1/2}), t^{1/2} \} $. Observing  Corollary \ref{sups}, we construct a $m$-superharmonic function $ \tilde{v} \in \Cc(\Omb) $ such that  $ \tilde{v} = \fii $ on $\Omf $ and $ \omega_{\tilde{v}}(t) \leq \lambda \max \{ \omega_\fii (t^{1/2}), t^{1/2} \}, $
where 
$ \lambda= \eta (1+ \|f \|^{1/m}_{L^\infty(\Omf)})$ and $ \eta >0 $ depends on $ \Om$.\\
Applying the comparison principle (see \cite{Bl05,Lu12}), we get that 
 $$ v(z) \leq \U(z) \leq \tilde{v}(z) \; \text{for all} \; z \in \Omb . $$
Let us set $ g(t) = \max ( \lambda \max \{\omega_\fii(t^{1/2}),t^{1/2} \}, \omega_{f^{1/m}}(t) )  $, 
then 
$$ | \U(z) - \U(\xi) | \leq g(| z-\xi |) ; \forall z \in \Om , \forall \xi \in \Omf $$ 
because,
$$ -g( | z-\xi |) \leq v(z) - \fii(\xi) \leq \U(z) - \fii(\xi) \leq \tilde{v}(z) - \fii(\xi) \leq g(| z-\xi |). $$ 
Let us fix a point $z_0 \in \Om$, For any small vector $\tau \in \C^n $, we define
\begin{center}
$ V(z,\tau)= \left\{ \begin{array}{ll}
                   \U(z) & ; z+\tau \notin \Om , z \in \Omb \\
                   max(\U(z),v_1(z)) & ; z,z+\tau \in \Om  \\ 
\end{array} \right. $ \\
\end{center}
where $v_1(z)= \U(z+\tau)+ g(|\tau |) | z - z_0 |^2 - d^2 g(| \tau |) -  g(| \tau |) $ . \\
Observe that if  $ z \in \Om , z+\tau \in \partial \Om $, we have 
 $$ v_1(z) - \U(z) \leq g(| \tau |)   + g(| \tau |) | z -z_0 |^2 - A g(| \tau |) - g( | \tau |) \leq  0  \;\; (*) $$
hence $ v_1(z) \leq  \U(z)$ for   $ z \in \Om , z+\tau \in \partial \Om $. In particular, $V(z,\tau)$ is well defined  and belongs to  $SH_m(\Om) \cap \Cc(\Omb)$.\\
We assert that  $L_\al V \geq f^{1/m} $ for all $ \al_i \in _m $ , indeed 

\begin{tabular}{ll}
   $ L_\al v_1(z) $ & $\geq f^{1/m}(z+\tau) +  g(| \tau |) L_\al(| z -z_0 |^2)$ \\
   & $\geq f^{1/m}(z+\tau) +  g(| \tau | )$  \\
   & $\geq f^{1/m}(z+\tau) + | f^{1/m}(z+\tau) - f^{1/m}(z) | $ \\
   & $\geq f^{1/m}(z)$, \\
\end{tabular} 
\\
for all $\al_i \in \Sigma_m , 1\leq i \leq m-1 $.
\\
If  $ z \in \Omf $ , $ z+\tau \notin \Om $, then  $ V(z,\tau)=\U(z) = \fii(z)$; On the other hand,
 $ z \in \Omf , z+\tau \in \Om $ , by $(*)$ we get $ V(z,\tau)= max(\U(z),v_1(z))=\U(z) = \fii(z) $.
Then  $V(z,\tau)=\fii(z)$ on  $ \Omf $, hence $ V \in \fami$.\\
Consequently, $V(z,\tau) \leq \U(z) ; \forall z \in \Omb $.
This implies that if $ z \in \Om $ , $z+\tau \in \Omb $, we have
$$ \U(z+\tau)+ g(|\tau |) | z -z_0 |^2 - d^2 g(| \tau |) -  g(| \tau |) \leq \U(z) $$
Hence,
$$\U(z+\tau) - \U(z) \leq (d^2 + 1) g(| \tau |) -  g(| \tau |) . | z -z_0|^2  \leq (d^2 + 1) g(| \tau |).$$
Reversing the roles of  $ z+\tau $ and $ z $, we get 
 $$ | \U(z+\tau) - \U(z) | \leq (d^2+1) g(| \tau | ) ; \forall z \in \Omb .$$
Thus,
 $$ \omega_{\U} (t) \leq (d^2+1) g(t).  $$
 Finally
 $$ \omega_{\U} (t) \leq \gamma  \max \{ \omega_\fii(t^{1/2}), \omega_{f^{1/m}}(t), t^{1/2} \} , $$
 where $\gamma:= \tau (1+ \|f\|^{1/m}_{L^\infty (\Omb)}) $ and $\tau \geq 1$ is a constant depending on $\Om$.
\end{proof}

Theorem~\ref{main} has the following consequence.

\begin{corollary}\label{holder}
Let $\Om$ be a smoothly bounded strongly $m$-pseudoconvex  domain in $\C^n$. Let $ \fii \in \Cc^{2 \alpha } (\Omf) $ and $0 \leq f^{1/m} \in \Cc^\alpha (\Omb) $, $ 0< \alpha \leq 1/2 $. Then 
the solution of Dirichlet problem $\U$ belongs to $ \Cc^\alpha(\Omb) $.

\end{corollary}
This result was proved  by Nguyen in \cite{N13} for the homogeneous case ($ f\equiv 0 $) (see also \cite{Lu12,Lu13b}).

We now give examples to point out that there is a loss in the regularity up to the boundary and show that our result is optimal.

\begin{example}
Let $ \psi$ be a concave modulus of continuity on $[0,1]$ and
 $$ \fii(z) = -\psi[\sqrt{(1+Rez_1)/2}], \text{ for  }  z=(z_1,z_2,..., z_n) \in \partial\B \subset \C^n. $$
We can show that $ \fii \in \Cc(\partial \B)$ with modulus of continuity 
$ \omega_{\fii}(t) \leq C \psi(t)$
for some $C>0$.
\\
We consider the following Dirichlet problem: 
\begin{center}
$ \left\{\begin{array}{ll}
            u\in SH_m(\Om)\cap\Cc(\Omb)& \\
           (dd^c u)^m \wedge \beta^{n-m} =0   &in \;\;  \B \\
            u=\fii  &on \;\; \partial \B ,\\
         \end{array} \right. $
\end{center}
where $ 2 \leq m \leq n $ be an interger. Then by the comparison principle, $ \U(z):= -\psi[\sqrt{(1+Rez_1)/2}]$ is the unique solution to this problem.  One can observe by a radial approach to the boundary point $ (-1,0,...,0)$ that 
$$ C_1 \psi(t^{1/2}) \leq \omega_\U(t) \leq C_2  \psi(t^{1/2}),$$
for some $ C_1,C_2>0$.

\end{example}

\section{H\"older continuity of the solution when $ f \in L^p(\Om)$.}

\subsection{Preliminaries and known results}

The existence of a weak solution to complex Hessian equation in domains of $\C^n$ was established in the work of Dinew and Kolodziej \cite{DK11}. More precisely, let $\Om \Subset \C^n$ be a smoothly $(m-1)$-pseudoconvex domain, $\fii \in \Cc(\Omf)$ and $ 0 \leq f \in L^p(\Om) $ for some $ p > n/m $, then there exist $ \U \in SH_m(\Om) \cap \Cc(\Omb)$ such that $ (dd^c \U )^m \wedge \beta^{n-m} = f \beta^n$ in $\Om$  and $ \U = \fii $ on $\Omf$.

Recently, N.C. Nguyen in \cite{N13} proved that the H\"older continuity of this solution under some additional conditions on the density near the boundary and on the boundary data, that is for $ f \in L^p(\Om) $, $ p > n/m$ bounded near the boundary $\Omf$  or $ f \leq C |\rho|^{-m\nu} $ there and $ \fii \in \Cc^{1,1}(\Omf)$.

Here we follow the approach proposed in \cite{GKZ08} for the complex Monge-Amp\`ere equation. A crucial role in this approach  is played by an a priori weak stability estimate of the solution.
This approach has been adapted to the complex Hessian equation in \cite{N13} and \cite{Lu12}. Here is the precise statement. 
 
\begin{theorem}
Fix $ 0 \leq f \in L^p(\Om)$, $p>n/m$. Let $ \fii, \psi \in SH_m(\Om) \cap  L^{\infty}(\Omb) $ be such that $ (dd^c \fii)^m \wedge \beta^{n-m} = f \beta_n$ in $\Om$, and let $ \fii \geq \psi $ on $\Omf$. Fix $ r \geq 1 $ and $ 0 \leq \gamma < \gamma_r$, where $\gamma_r= \frac{r}{r+mq+\frac{pq(n-m)}{p- \frac{n}{m}}} $ and $1/p+1/q=1$. Then there exists a uniform constant $C=C(\gamma, \|f\|_{L^p(\Om)}) >0$ such that 
$$ \sup_\Om (\psi - \fii) \leq C ( \| (\psi - \fii )_+ \|_{L^r(\Om)} )^\gamma $$
where $ (\psi - \fii )_+ := max(\psi-\fii , 0)$.

\end{theorem}

The second result gives the H\"older continuity under some additional hypothesis.

\begin{theorem}(\cite{N13}).\label{Lp}
Let $ 0  \leq f \in L^p, p>n/m $ and $ \fii \in \Cc(\Omf)$. Let $\U$ be the continuous solution to \ref{hess}. Suppose that there exists $ v \in SH_m(\Om) \cap \Cc^{0,\nu}(\Omb) $, for $ 0 < \nu < 1$ such that $ v \leq \U $ in $ \Om$ and $ v=\U$ on $\Omf$.
\\
1) If $\nabla \U \in L^2(\Om)$ then $ \U \in \Cc^{0,\al} (\Omb) $ for any $\al < \min \{\nu, \gamma_2 \}$.
\\
2) If the total mass of $\Delta \U$ is finite in $\Om$ then $ \U \in \Cc^{0,\al} (\Omb) $ for any $\al < \min \{\nu, 2 \gamma_1 \}$,
where $ \gamma_r=  \frac{r}{r+mq+\frac{pq(n-m)}{p- \frac{n}{m}}} $ for $ r \geq 1$.
\end{theorem}

\subsection{Construction of H\"older barriers}

The remaining problem is to construct a H\"older continuous barrier with the right exponent which guarantees one of the conditions in  Theorem~\ref{Lp}.

 Using  the interplay between the real and the complex Monge-Amp\`ere measures as suggested by Cegrell and Persson in \cite{CP92},  we will construct H\"older continuous $m$-subharmonic barriers for the problem (\ref{hess})  when $ f \in L^p(\Om)$, $ p \geq 2n/m$.
 
We recall that if $u$ is a locally convex smooth function in $\Om$, its real Monge-Ampere measure is defined by 
 $$ M u := det \left( \frac{\partial^2 u }{\partial x_j \partial x_k} \right) d V_{2 n}. $$
When $u$ is only convex then $ M u$ can be defined following Alexandrov  \cite{A55} by means of the gradient image as a nonnegative Borel measure on $\Om$ (see \cite{Gut01}, \cite{RT77}, \cite{Gav77}).

\begin{proposition}\label{real}
Let $ 0 \leq f \in L^p(\Om)$,  $p \geq 2n/m$ and $ u $  be a locally convex function in $ \Om$ and continuous in $ \Omb$. If the real Monge-Amp\`ere measure $ M u \geq f^{2n/m} d V_{2 n} $ then the complex Hessian measure satisfies the inequality $ (dd^c u)^m \wedge \beta^{n-m} \geq f \beta^n $ in the weak sense of measures on $\Omega$.

\end{proposition}

\begin{proof}
First step, we claim that 
$$ (dd^c u)^n \geq f^{n/m} \beta^n \text{in the sense of measures. }$$
Indeed, for a smooth function $u$, we have
\begin{equation}
 | det(\partial^2 u / \partial z_j \partial \bar{z}_k)|^2 \geq det (\partial^2 u/ \partial x_j \partial x_k )  
\label{dddd}
\end{equation}
hence we get immediately that $ (dd^c u)^n \geq f^{n/m} \beta^n $ (see \cite{CP92}).
\\
Moreover, it is well known  for smooth convex function that 
\begin{equation}
(Mu )^{1/n} = \inf \D u ,\text{   where  } \Delta_H u :=  \sum_{j,k} h_{jk} \frac{\partial^2u}{\partial x_j \partial x_k}.
\label{ddd}
\end{equation}
for any symmetric positive definite matrix  $H = (h_{jk})$ with $det H= n^{-n}$ (see \cite{Gav77}, \cite{Bl97}).
In general case, we will prove that $ (dd^c u)^n \geq f^{n/m} \beta^n$ weakly in $\Om$. Indeed, the problem being local, we can assume that $u$ is defined and convex in a neighborhood of a ball $\bar 
B \subset \Om$. For $\delta >0$ we put
 $ \mu_\delta := Mu* \rho_\delta $ then $ \mu_\delta \geq g_\delta $ where $ g_\delta := f^{2n/m} * \rho_\delta $ (without loss of generality we assume $ g_\delta >0$), we may assume that $u$ and $\mu_\delta$ are defined in this neighborhood of $\bar{B}$. Let $\fii_\delta$ be a sequence of smooth function on $ \partial B$ converging uniformly to $u$ there.
 Let $u^\delta$ is a smooth convex function such that $M u^\delta = \mu_\delta $ in $B$ and $u^\delta = \fii_\delta $ on $\partial B$.
 Let $ \tilde{u}$ convex continuous on $\bar{B}$ such that $ M \tilde{u} =0$ and $\tilde{u}= \fii_\delta$ on $ \partial B$. Moreover, let $ v^\delta$ convex continuous on $\bar{B}$ such that $ M v^\delta =\mu_\delta$ and $v^\delta= 0$ on $ \partial B$.
\\
From the comparison principle for the real Monge-Amp\`ere operator (see \cite{RT77}), we can see that
\begin{equation}
\tilde{u}+ v^\delta \leq u^\delta \leq \tilde{u} - v^\delta . 
\label{d}
\end{equation}
It follows from Lemma 3.5 in \cite{RT77} that
\begin{equation}
 (-v^\delta (x))^n \leq c_n (diam B)^{n-1} dist(x,\partial B) M v^\delta (B), \;\; x \in B.
\label{dd}
\end{equation}
Then we conclude that $\{u^\delta \}$ is uniformly bounded on $B$, hence there exists a subsequene $ \{u^{\delta_j}\} $ converging locally uniformly on $B$.\\
Moreover, \eqref{d} and \eqref{dd} imply that $\{u^{\delta_j}\} $ is uniformly convergent on $\bar{B}$.
From the comparison principle it follows that $ u^{\delta_j}$ converges uniformly to $u$.
Since $ u^{\delta_j} \in \Cc^{\infty} (\bar{B})$ and $ M u^{\delta_j} \geq f^{2n/m} * \rho_{\delta_j} d V_{2 n} $, we get that 
$$ (dd^c u^{\delta_j} )^n \geq (f^{2n/m} * \rho_{\delta_j})^{1/2} \beta^n$$
Finaly, as $u^{\delta_j}$ converges uniformly to $ u $ and by the convergence theorem of Bedford and Taylor, we get that
$$ (dd^c u)^n \geq f^{n/m} \beta^n . $$
Second step, let $ v = |z|^2 \in PSH(\Om)$. Since $ (dd^c u)^n \geq f^{n/m} \beta^n $ and $ (dd^c v)^n = \beta^n$, by Theorem 1.2 in \cite{Di09} we get that
$$ (dd^c u)^m \wedge \beta^{n-m} \geq f \beta^n .$$
\end{proof}

We recall the theorem of existence of convex solution to the Dirichlet problem for the real Monge-Amp\`ere equation, this theorem is due to Rauch and Taylor.

\begin{theorem} (\cite{RT77}). \label{taylor}
Let $\Om$ is strictly convex domain. Let $\fii \in \Cc(\Omf)$ and $\mu$ is a non negative Borel measure on $\Om$ with $ \mu(\Om) < \infty $. Then there is a unique convex $u \in \Cc(\Omb)$ such that 
$ M u = \mu $ in $\Om$ and $ u = \fii$ on $\Omf$.
\end{theorem}

The following result gives the existence of a $1/2$-H\"older continuous $m$-subharmonic barrier for the problem (\ref{hess}) when $f \in L^p(\Om)$, $p\geq 2n/m$.

\begin{theorem}\label{barrier}
Let $\fii \in \Cc^{0,1}(\Omf)$ and $ f \in L^p(\Om)$, $ p  \geq 2n/m$. Then there exists $ v \in SH_m(\Om) \cap \Cc^{0,1/2}(\Omb) $ such that $ v = \fii $ on  $ \Omf$ and $ (dd^cv)^m \wedge \beta^{n-m} \geq f \beta^n $ in the weak sense of currents, hence   $ v \leq \U $ in $\Om$.
\end{theorem}

\begin{proof}
Let $B$ a big ball contains $\bar \Om$ and let $ \tilde{f} $ be the function defined  by $ \tilde{f}=f$ on $\Om$ and $\tilde{f}=0 $ on $B \setminus \Om$. Then $\tilde{f}\in L^p(B)$, $p\geq 2n/m$ .
Let $ \mu := \tilde{f}^{2n/m} \beta^n$. This is a nonnegative Borel measure on $B$ with $\mu(B) < \infty$. Thanks to  Theorem \ref{taylor} there exists a unique convex function $u \in \Cc(\bar{B})$ such that 
$ M u = \mu $ in $B$ and $ u= 0$ on $\partial B$. Hence $u$ is lipschitz continuous on $\Omb$. By Proposition \ref{real} we have $ (dd^c u)^m \wedge \beta^{n-m} \geq f \beta^n $ in $\Om$.\\
We will construct the required barrier as follows. Let $ h_{\fii - u} $ be the upper envelope of $  \mathcal{V}_m(\Om,\fii -u,0)$. Then, thanks to Theorem \ref{main},  $ h_{\fii-u} $ is H\"older continuous of exponent  $1/2$ in $\Omb$.
Now  it is easy to check that $ v:= u + h_{\fii-u}$ is $m$-sh in $\Om$ and satisfies $ v = \fii $ in  $ \Omf$ and $ (dd^cv)^m \wedge \beta^{n-m} \geq f \beta^n $ on $\Om$,  hence $ v \leq \U $ in $\Om$.
\end{proof}

The last theorem provides us with a H\"older continuous barrier for the Dirichlet problem (\ref{hess}) with better exponent.

However, when $ f \in L^p(\Om)$ for $ p > n/m$, we can find a H\"older continuous barrier with small exponent less than $\gamma_1$.

\begin{proposition}\label{barrier2}
Let $\fii \in \Cc^{0,1}(\Omf)$ and $ f \in L^p(\Om)$, $ p  > n/m$. Then there exists $ v \in SH_m(\Om) \cap \Cc^{0,\al}(\Omb) $ for $ \al < \gamma_1= \frac{1}{1+mq+\frac{pq(n-m)}{p- \frac{n}{m}}} $ such that $ v = \fii $ on  $ \Omf$ and $ v \leq \U $ in $\Om$.
\end{proposition}
 
\begin{proof}
Let us fix a large ball $ B \subset \C^n$ such that $ \Om \Subset B \subset \C^n$. We define $ \tilde{f} =f  $  in $\Om$ and $ \tilde{f}=0$ in $B \setminus \Om$. Let  $h_1$ to the Dirichlet problem in $B$ with density $ \tilde{f}$ and zero boundary values.  Since $ \tilde{f}\in L^p(\Om)$ is bounded near $ \partial B$,  $h_1$  is H\"older continuous  on $ \bar{B}$ with exponent $ \alpha_1 < 2\gamma_1$ (see \cite{N13}). Now let $h_2$ denote the solution to the Dirichlet problem in $\Om$ with boundary values $\fii-h_1$ and the zero density.
Thanks to Theorem \ref{main}, we see that $ h_2 \in \Cc^{0,\alpha_2}(\Omb)$ where $ \alpha_2= \alpha_1/2$.
Therefore, the required barrier will be $ v =  h_1 + h_2$. It is clear that $ v \in SH_m( \Om) \cap \Cc(\Omb)$, $ v|_{\Omf} = \fii$ and $ (dd^c v)^m \wedge \beta^{n-m} \geq f \beta^n$ in the weak sense in $\Om$. Hence, by the comparison principle we get that $ v \leq \U  $ in $\Om$ and $ v=\U=\fii$ on $\Omf$. Moreover we have $ v \in \Cc^{0,\al}(\Omb)$ for any $ \alpha < \gamma_1$. 
\end{proof}

\noindent
We will need in the sequel the  following elementary lemma.
\begin{lemma}\label{comparison}
Let $ u,v \in SH_m(\Om) \cap \Cc(\Omb)$ such that $ v \leq u$ on $\Om$ and $u=v$ on $\Omf$. Then 
$$ \int_{\Om} dd^c u \wedge \beta^{n-1} \leq \int_{\Om} dd^c v \wedge \beta^{n-1}.$$
\end{lemma} 

We recall the definition of the class $ \mathcal{E}_m^0(\Om)$ (see \cite{Lu13c}).
\begin{definition}
We denote  $ \mathcal{E}_m^0(\Om)$ the class of  bounded functions $v$ in $SH_m(\Om)$ such that 
$ \lim_{z\to \Omf} v(z) =0$ and $ \int_\Om (dd^c v)^m \wedge \beta^{n-m} < + \infty$.
\end{definition}
The following proposition was proved by induction in \cite{Ce04} for plurisubharmonic functions and we can do the same argument for $m$-sh fucntions.
\begin{proposition}
Suppose that $ h,u_1,u_2 \in  \mathcal{E}_m^0(\Om)$, $ p,q \geq 1$ such that $ p+q \leq m$ and $ T= dd^c g_1 \wedge ... \wedge dd^c g_{m-p-q} \wedge \beta^{n-m} $ where $ g_1, ..., g_{m-p-q} \in 
 \mathcal{E}_m^0(\Om) $. Then we get
 $$
  \int_\Om -h (dd^c u_1)^p \wedge (dd^c u_2)^q \wedge T \leq \left[   \int_\Om -h (dd^c u_1)^{p+q} \wedge T \right]^{\frac{p}{p+q}} \left[   \int_\Om -h (dd^c u_2)^{p+q} \wedge T \right]^{\frac{q}{p+q}}.
 $$
\end{proposition}

\begin{proof}
We first prove the statement for $p=q=1$. Thanks to the Cauchy-Schwarz inequality (see \cite{Lu13c}), we have
$$
\begin{array}{ll}
  \int_\Om -h dd^c u_1 \wedge dd^c u_2  \wedge T & =  \int_\Om -u_1 dd^c u_2 \wedge dd^c h \wedge T   \\
  & \leq \left[    \int_\Om -u_1 dd^c u_1 \wedge dd^c h   \wedge T \right]^{1/2} \left[    \int_\Om -u_2 dd^c u_2 \wedge dd^c h   \wedge T \right]^{1/2} \\
  & = \left[    \int_\Om -h (dd^c u_1)^2     \wedge T \right]^{1/2} \left[    \int_\Om -h (dd^c u_2)^2 \wedge  T \right]^{1/2} .\\
\end{array}
$$
The general case follows by induction in the same way as in \cite{Ce04}.  
\end{proof}
We will only need the following particular case.
\begin{corollary}\label{finite lap}
Let $u_1,u_2 \in  \mathcal{E}_m^0(\Om)$. Then we have
$$
 \int_\Om  dd^c u_1 \wedge (dd^c u_2)^{m-1} \wedge \beta^{n-m} \leq \left[   \int_\Om  (dd^c u_1)^{m} \wedge \beta^{n-m} \right]^{\frac{1}{m}} \left[   \int_\Om  (dd^c u_2)^{m} \wedge \beta^{n-m} \right]^{\frac{m-1}{m}}.
$$
\end{corollary}
\subsection{\bf{Proof of Theorem \ref{main2}}}

Let $\U_0$ the solution to the Dirichlet problem (\ref{hess}) with zero boundary values and the density $f$. 
We first claim that the total mass of $ \Delta \U_0$ is finite in $\Om$. Indeed, let $ \rho$ be the  defining function of $\Om$, then by Corollary \ref{finite lap} we have that
\begin{equation}\label{cegrell inequlaity}
 \int_\Om  dd^c \U_0 \wedge (dd^c \rho)^{m-1} \wedge \beta^{n-m} \leq \left[   \int_\Om  (dd^c \U_0)^{m} \wedge \beta^{n-m} \right]^{\frac{1}{m}} \left[   \int_\Om  (dd^c \rho)^{m} \wedge \beta^{n-m} \right]^{\frac{m-1}{m}}.
\end{equation}  
Since $\Om$ is a bounded strongly $m$-pseudoconvex domain, there exists a constant $c>0$ such that $ (dd^c \rho)^j \wedge \beta^{n-j} \geq c \beta^n$ in $\Omega$ and we can find $A \gg 1$ such that $ A \rho - |z|^2$ is $m$-sh function. Now it is easy to see that 
$$
\int_\Om dd^c \U_0 \wedge \beta^{n-1} \leq \int_\Om dd^c \U_0 \wedge (A dd^c \rho)^{m-1} \wedge \beta^{n-m} .
$$
 Hence the inequality \ref{cegrell inequlaity}   yields 
$$
\int_\Om dd^c \U_0 \wedge \beta^{n-1} \leq A^{m-1}  \left[   \int_\Om  (dd^c \U_0)^{m} \wedge \beta^{n-m} \right]^{\frac{1}{m}} \left[   \int_\Om  (dd^c \rho)^{m} \wedge \beta^{n-m} \right]^{\frac{m-1}{m}}.
$$

\noindent
Now we note that the total mass of complex Hessian measures of $\rho$ and $ \U_0$ are finite in $\Om$.
Therefore, the total mass of $\Delta \U_0$ is finite in $\Om$.
\\
Let $\tilde{\fii}$ be a $\Cc^{1,1}$-extension of $\fii$ to $\Omb$ such that $ \| \tilde{\fii}\|_{\Cc{1,1}(\Omb)} \leq C \| \fii \|_{\Cc^{1,1}(\Omf)}$ for some $C>0$.
Now, let $ v= A\rho + \tilde{\fii} + \U_0$ where $ A \gg 1$ such that $ A \rho +\tilde{\fii} \in SH_m(\Om) \cap \Cc(\Omb)$. By the comparison principle we see that $ v \leq \U$ in $\Om$ and $ v=\U =\fii$ on $\Omf$. Since $ \rho$ is smooth in a neighborhood of $ \Omb$ and $ \| \Delta \U_0 \|_\Om < +\infty$, we get that $  \| \Delta v \|_\Om < +\infty$.
Then  by Lemma \ref{comparison} we have $ \| \Delta\U \|_\Om < + \infty$.

\noindent
When $p > n/m$,  we can get  by  Theorem \ref{barrier2} and Theorem \ref{Lp} that
$\U \in  \Cc^{0,\al}(\Omb)$ where $ \alpha < \gamma_1 $.

\noindent
Moreover, if $p \geq 2n/m$, the Proposition \ref{barrier} gives the existence of a $1/2$-H\"older continuous barrier to the Dirichlet problem. Then using Theorem \ref{Lp} we obtain that   
$\U \in  \Cc^{0,\al}(\Omb)$ where $ \alpha <\min \{ 1/2, 2 \gamma_1 \}$.
\\

Finally, in the particular case when  $ f \in L^p(\Om)$, for $ p > n/m$ and  satisfies some  condition near the boundary $\Omf$, we obtain a better exponent.

\begin{proposition}
Let $\Om \subset \C^n$ be a  strongly $m$-pseudoconvex bounded domain with smooth boundary, suppose that  $ \fii \in \Cc^{1,1} (\Omf) $ and $ 0 \leq f \in L^p(\Om)$, for some $ p > n/m $,  and
$$ f(z) \leq (h\circ \rho (z))^m \text{  near } \Omf,$$
where $\rho$ is the defining function on $\Om$ and $ 0 \leq h \in L^2([-A,0[)$, with $ A \geq \sup_{\Om} |\rho| $, be an increasing function. Then the solution $\U$ to (\ref{hess}) is H\"older continuous with exponent $ \alpha <   \gamma_2 $.
\end{proposition}

\begin{proof}
Let $\chi:[-A,0] \to \R^-$ be the primitive of $h$ such that $\chi (0) = 0$. It is clear that $ \chi$ is a convex increasing function. By the H\"older inequality, we  see that
$$ |\chi(t_1) - \chi(t_2)| \leq \|h\|_{L^2 } |t_1-t_2|^{1/2},$$ 
for all $ t_1,t_2 \in [-A,0]$.
From the hypothesis, there exists a compact $K \Subset \Om$ such that 
\begin{equation}\label{f}
 f(z) \leq (h \circ \rho (z))^m  \text{ for }  z \in \Om \setminus K .
\end{equation}

Then the function $ v = \chi \circ \rho $ is m-subharmonic in $\Om$, continuous in $\Omb$ and  satisfies 
$$ dd^c \chi \circ \rho = \chi'' (\rho) d\rho \wedge d^c \rho + \chi'(\rho) dd^c \rho \geq \chi'(\rho) dd^c \rho,$$ in the sense of currents on $\Omega$.

From the definition of $ \rho$ (see the Definition \ref{def}), there is a constant $A>0$ such that $ (dd^c \rho)^m \wedge \beta^{n-m} \geq A \beta^n$, hence the inequality (\ref{f}) yields 
\begin{equation}\label{ff}
 (dd^c v)^m \wedge \beta^{n-m} \geq A  (h \circ \rho)^m \beta^n \geq A. f \beta^n \text{   in  } \Om \setminus K .
\end{equation}
Now consider a $\Cc^{1,1}$ extension  $ \tilde \fii $ of $\fii$ to $\Omb$ and choose $ B \gg 1 $ large enough so that 
$\tilde \fii + B \rho $ is $m$-subharmonic in $\Om$ and 
$$ 
\tilde{v} := B v+ \tilde \fii + B \rho \leq \U \text{ in a neighborhood of } K.
$$ 
Then  $\tilde{v}$ is $m$-subharmonic in $\Omega$ and if  $ B \geq 1/A$, it follows from (\ref{ff}) that

$$ 
(dd^c \tilde{v})^m \wedge \beta^{n-m} \geq f \beta^n \text{ in } \Om \setminus K. 
$$
By the comparison principle (see \cite{Bl05, Lu12}), we have $ \tilde{v} \leq \U $ on $\Om \setminus K$, hence we get 
$ \tilde{v} \leq \U \text{  on   } \Omb,$ $ \tilde{v} = \fii $ on $ \Omf$ and $ \tilde{v} \in \Cc^{0,1/2}(\Omb)$.
\\
We claim that $ \nabla \tilde{v} \in L^2(\Om)$. Indeed, it is enough to observe that   $\tilde \fii + B \rho $ is Lipschitz in $\Omb$ and 
$$ 
\int_{\Om}  dv \wedge d^c v \wedge \beta^{n-1} = \int_{\Om} (h \circ \rho)^2 d\rho \wedge d^c \rho \wedge \beta^{n-1} < + \infty, 
$$
since   $ h\circ \rho \in L^2(\Om)$.

Therefore  $ \tilde{v} \leq \U $ and  $ \tilde{v}=\fii $ on $\Omf$. Then an easy integration by parts shows that 
$$
 \int_{\Om} d \U \wedge d^c \U \wedge \beta^{n-1}  \leq  \int_{\Om}d \tilde{v} \wedge d^c \tilde{v} \wedge \beta^{n-1} < + \infty,
$$
hence $ \nabla \U \in L^2(\Om)$ (see \cite{GKZ08},  \cite{N13}).

By Theorem \ref{Lp}, we get that $ \U \in \Cc^{0,\alpha}(\Omb)$ for any 
$\alpha < \min \{ 1/2, \gamma_2 \} =  \gamma_2.$

\end{proof}
As an example of application  of the last result, fix $p > n \slash m$, take $h (t) := t^{- \alpha}$ with $0 < \alpha < 1 \slash (pm)$, $ t<0$ and define  $f := (h \circ \rho)^m$.

\subsection{H\"older continuity for radially symmetric solution}

Here we consider the case when the right hand side and the boundary data are radial. 
In this case, Yong and Lu \cite{YL10} gave an explicit formula for the radial solution of the Dirichlet problem (\ref{hess}) with $ f \in \Cc(\B)$. Moreover, they studied higher regularity for radial solutions (see also \cite{DD12}).
\\

Here, we will extend this explicit formula to the case when $ f \in L^p(\B)$, for $ p>n/m$, is a radial non-negative  function  and $ \fii \equiv 0 $ on $ \partial \B$. Then we prove  H\"older continuity of the radially symmetric solution with a better exponent which turns out to be optimal.

\begin{theorem} \label{thm:radial}
Let $ f \in L^p(\B)$  be a radial function, where $p > n\slash m$. Then the unique solution $\U$ for (\ref{hess}) with zero boundary value is given by the explicit  formula 
\begin{equation}\label{radial}
 \U(r) = - B \int_{r}^{1} \frac{1}{t^{2n/m-1}} \left( \int_{0}^{t} \rho^{2n-1} f(\rho) d \rho \right)^{1/m} dt,
\end{equation}
where $ B= \left( \frac{C^m_n}{2^{m+1}n}  \right)^{-1/m}$.
Moreover, $ \U \in \Cc^{0,2-\frac{2n}{mp}} (\bar{\B})$ for $ n/m < p<2n/m $ and  $ \U \in Lip(\bar{\B})$ for $ p \geq 2n/m$.

\end{theorem}

\begin{proof}
Let $ f_k \in \Cc(\bar{\B})$ positive radial symmetric function such that $\{f_k\}$ converges to $f$ in $L^p(\B)$. Then there exists a unique solution  $ \U_k \in \Cc(\bar{\B})$ for $Dir(\B,0,f_k)$ (see \cite{YL10}) given by  the following formula:

$$ 
\U_k(r)= -B \int_{r}^{1} \frac{1}{t^{2n/m-1}} \left( \int_{0}^{t} \rho^{2n-1} f_k(\rho) d \rho \right)^{1/m} dt.
$$
It is clear that $\U_k$ converges in $L^1(\B)$ to  the function $\tilde{u}$ given by the same formula i.e.
$$ 
\tilde{u}(r)= - B \int_{r}^{1} \frac{1}{t^{2n/m-1}} \left( \int_{0}^{t} \rho^{2n-1} f(\rho) d \rho \right)^{1/m} dt.
$$
We claim that  the sequence $ \{ \U_k\}$ is uniformly bounded and equicontinuous in $\bar \B$. Indeed, let $ 0<r<r_1 \leq 1$, we have
\begin{center}
$
\begin{array}{ll}
|\U_k(r_1)-\U_k(r)| & =  B \int_{r}^{r_1} \frac{1}{t^{2n/m-1}} \left( \int_{0}^{t} \rho^{2n-1} f_k(\rho) d \rho \right)^{1/m} dt \\
& \leq  B \int_{r}^{r_1} \frac{1}{t^{2n/m-1}} \left( \int_{0}^{t} \rho^{(2n-1)/q} \rho^{(2n-1)/p} f_k(\rho) d \rho \right)^{1/m} dt \\
& \leq C \| f_k\|^{1/m}_{L^p(\B)} \int_{r}^{r_1} \frac{1}{t^{2n/m-1}} \left( \int_{0}^{t} \rho^{2n-1}  d \rho \right)^{1/mq} dt \\
& \leq C \|f_k \|^{1/m}_{L^p(\B)} (r_1^{2-\frac{2n}{mp}} - r^{2-\frac{2n}{mp}}). \\

\end{array}
$
\end{center}
Since $f_k $ converges to $ f$ in $L^p(\B)$, we get $ \| f_k\|_{L^p(\B)} \leq C $ where $ C>0 $ does not depend on $k$, hence $\U_k$ is equicontinuous on $ \bar{\B}$. By Arzel\`a-Ascoli theorem, there exists a subsequence $ \U_{k_j}$ converges uniformly to $\tilde{u}$.\\
Consequently, $ \tilde{u} \in SH_m(\B) \cap \Cc(\bar{\B})$ and thanks to the convergence theorem for the Hessian operator (see \cite{Lu12}) we can see that $ (dd^c \tilde{u})^n = f \beta^n$ in $\B$.
\\
Passing to the limit in the inequality 
$$  |\U_k(r_1)-\U_k(r)|  \leq C \|f_k \|^{1/m}_{L^p(\B)} (r_1^{2-\frac{2n}{mp}} - r^{2-\frac{2n}{mp}}),$$
we get that
$$  |\tilde{u}(r_1)-\tilde{u}(r)|  \leq C \|f \|^{1/m}_{L^p(\B)} (r_1^{2-\frac{2n}{mp}} - r^{2-\frac{2n}{mp}}).$$
Hence, for $ p \geq 2n/m $ we get
$ \tilde{u} \in Lip (\bar{\B})$
and for $ n/m < p< 2n/m $, we have
$\tilde{u} \in \Cc^{0,2-\frac{2n}{mp}}(\bar{\B})$.

\end{proof}

We give an example which illustrates that the H\"older  exponent $2-\frac{2n}{mp}$ given by the Theorem~\ref{thm:radial} is optimal.

\begin{example}
Let $p \geq 1$ a fixed exponent. Take $ f_{\alpha}(z) = \frac{1}{ \vert z \vert^\alpha} $, with $ 0 < \alpha < 2n \slash p$. Then it is clear that $ f_{\alpha} \in L^p(\B)$.
The unique radial solution to the  Dirichlet problem (\ref{hess}) with right hand side $f_{\alpha}$ and zero boundary value is given by
 $$
  \U_{\alpha} (z) = c ( r^{2- \alpha \slash m} -1),  r := \vert z \vert \leq 1,
 $$
 where $ c = \left( \frac{C^m_n}{2^{m+1}n}  \right)^{-1/m} ( \frac{1}{2n-\alpha} )^{1/m} \frac{m}{2m-\alpha} $.
 Then we have 
 
 1. If $ p > n \slash m$ then $0 < \alpha < 2 m$ and the solution $\U_{\alpha}$ is $ (2-\frac{2n}{mp} + \delta)-$H\"older with $ \delta= (2n/p - \alpha)/m $. Since $\alpha$ can be choosen arbitrary close to $2 n \slash p$, this implies that the optimal H\"older exponent is $ 2-\frac{2n}{mp}$.

2. Observe that when $ 1 \leq p < n \slash m$ and $ 2m < \alpha < 2n $, then the solution $\U_{\alpha}$ is unbounded. 

\end{example}
The next example shows that  in Theorem~\ref{thm:radial},  $n /m$ is the critical exponent in order to have a continuous solution.
\begin{example}
Consider the density $f$ given by the formula
$$ f(z)  := \frac{1}{|z|^{2m} (1-log|z|)^{\gamma}},$$
where $\gamma> m \slash n$ is fixed.

It is clear that $ f \in L^{n/m} (\B) \setminus L^{n/m+\delta}(\B) $ for any $\delta>0$.
An elementary computation shows that the corresponding solution $ \U$ given by the explicit formula  (\ref{radial})
can be estimated by
$$ \U(z) \leq C (1-(1-log\vert z\vert)^{1-\gamma \slash m} ),$$
where $C>0$ depends only on $ n,m$ and $\gamma$. 
Hence we see that if $m /n < \gamma < m$ then $ \U$ goes to $ -\infty$ when $ z$ goes to $ 0$.
 In this case the solution $\U$ is unbounded.
\end{example}

\bigskip

\bigskip
\noindent {\bf Acknowledgements}.
I am greatly indebted to my advisor, Professor Ahmed Zeriahi, who has thoroughly read this paper. His guidance helped me in all the time of my study. I also wish to thank Hoang Chinh Lu for valuable discussions and encouragements.

\bigskip

\noindent 
Mohamad Charabati \\
Institut de Math\'ematiques de Toulouse \\
Universit\'e Paul Sabatier \\
118 route de Narbonne \\
31602 Toulouse Cedex 09 (France). \\
e-mail: {\tt mohamad.charabati@math.univ-toulouse.fr }


\bigskip
\bigskip 

\begin{thebibliography}{BCHM}

\bibitem[A55]{A55}{ A. D. Alexandrov, \it Die innere Geometrie der konvexen Flächen, \rm Akademie Verlag, Berlin, 1955}.

\bibitem [BT76] {BT76}{E. Bedford and B. A. Taylor, \it
  The Dirichlet problem for the complex Monge-Amp\`ere operator, \rm Invent. Math. \ 37 \rm (1976), 1-44.}

\bibitem [BT82]{BT82}{E. Bedford and B. A. Taylor, \it  A new capacity for plurisubharmonic functions, \rm Acta Math.  149 \rm (1982), 1-40.}
\bibitem[Bl96]{Bl96}{ Z. B\l ocki, \it The complex Monge-Amp\`ere operator in hyperconvex domains}, Ann. Scuola Norm. Sup. di Pisa, 23 (1996), 721-747.

\bibitem[Bl97]{Bl97}{ Z. B\l ocki, \it Smooth exhaustion functions in convex domains, \rm Proc. Amer. Math. Soc. 125(1997), 477-484.}

\bibitem[Bl05]{Bl05}{Z. B\l ocki, \it Weak solutions to the complex Hessian equation, \rm Ann. Inst. Fourier (Grenoble) 55, 5 (2005), 1735-1756.}

\bibitem[Ce04]{Ce04}{U. Cegrell, \it The general definition of the complex Monge-Amp\`ere operator, \rm Ann. Inst. Fourier,  54, n.1 (2004), 159-179.}

\bibitem[Ch14]{Ch14} {M. Charabati, \it H\"older regularity for solutions to complex Monge-Amp\`ere equations, preprint 2014.}

\bibitem[CK94]{CK94}{ U. Cegrell and S. Ko\l odziej, \it The Dirichlet problem for the complex Monge-Amp\`ere operator: Perron classes and rotation invariant measures, \rm Michigan Math. J.,
41(1994), no. 3, 563-569.}

\bibitem [CP92]{CP92}{U. Cegrell and L. Persson, \it  The Dirichlet problem for the complex Monge– Amp\`ere operator: Stability in L2. \rm Michigan Math. J. 39 (1992),145-151.}


\bibitem [D89]{D89} {J.-P. Demailly, \it Potential theory in several complex variables}, Lecture notes, ICPAM, Nice, 1989.

\bibitem[DD12]{DD12}{N.Q. Dieu and N.T. Dung, \it Radial symmetric solution of complex Hessian equation in the unit ball, \rm Complex Var. Elliptic Equ. 58, No. 9, (2013), 1261-1272 }

\bibitem[Di09]{Di09}{S. Dinew, \it An inequality for mixed Monge-Amp\`ere measures, \rm
Math. Zeit.262 (2009), 1-15}

\bibitem[DK11]{DK11}{ S. Dinew, S. Ko\l odziej, \it A priori estimates for complex Hessian equations, \rm preprint arxiv: 1112.3063v1.}

\bibitem [EGZ11]{EGZ11}{ P. Eyssidieux, V. Guedj and A. Zeriahi}, {\it Viscosity solutions to degenerate complex Monge-Amp\`ere equations}, \rm Comm. Pure Appl. Math. { 64} (2011), 1059-1094.

\bibitem[G59]{G59} {L. G{\aa}rding, \it An inequality for Hyperbolic Polynomials, \rm Journal of Mathematics and Mechanics, Vol. 8, No. 6 (1959).}

\bibitem[Gav77]{Gav77}{B. Gaveau, \it M\'ethodes de contr\^ole optimal en analyse complexe I. \rm R\'esolution d'\'equation de Monge-Amp\`ere. J.Funct. Anal. 25 (1977), no. 4, 391–411. }

\bibitem [GKZ08]{GKZ08}{V. Guedj, S. Ko\l odziej and A. Zeriahi, \it
     H\"older continuous solutions to the complex  Monge-Amp\`ere equations equations, \rm
     Bull. London Math. Soc. \bf 40 \rm (2008), 1070-1080.}

\bibitem[Gut01]{Gut01}{C. E. Guti\'errez, \it The Monge–Amp\`ere equation, \rm Birkh\"auser, Boston, MA, 2001}.

\bibitem[H94]{H94} {L. H\"ormander, \it Notions of convexity, \rm Birkhuser, Basel-Boston-Berlin, 1994.}

\bibitem[Ko98]{Ko98}{S. Ko\l odziej, \it The complex Monge-Amp\`ere equation, \rm Acta Math. 180 (1998) 69-117.}

\bibitem [Kor82]{Kor82}{ N. P. Korneichuk, \it Precise constant in Jackson's inequality for continuous periodic functions, \rm  Math. Zametki, 32 (1982), 669–674.}

\bibitem[Li04]{Li04}{S.-Y. Li, \it On the Dirichlet problems for symmetric function equations of the eigenvalues of the complex Hessian, \rm Asian J.Math. 8 (2004), 87-106.}

\bibitem[Lu12]{Lu12}{H.C. Lu, \it Equations Hessiennes complexes, \rm PhD Thesis defended on 30th
November 2012, http://thesesups.ups-tlse.fr/1961/ }.

\bibitem[Lu13a]{Lu13a}{H.C. Lu, \it Solutions to degenerate complex Hessian equations, \rm Journal de math\'ematiques pures et appliqu\'ees 100 (2013) pp. 785-805.}

\bibitem[Lu13b]{Lu13b}{H.C. Lu, \it Viscosity solutions to complex Hessian equations, \rm Journal of Functional Analysis 264, no 6, (2013) 1355-1379.}

\bibitem[Lu13c]{Lu13c}{H.C. Lu, \it A variational Approach to complex Hessian equations in $\C^n$, \rm preprint arXiv: 1301.6502v2.}


\bibitem[N13]{N13}{N.C. Nguyen, \it H\"older continuous solutions to complex Hessian equations, \rm preprint arXiv: 1301.0710v2.}

\bibitem[RT77]{RT77}{J. Rauch and B.A. Taylor, \it The Dirichlet problem for the multidimensional Monge-Amp\`ere equation, \rm Rocky Mountain Math. J. 7(1977), 345-364. }

\bibitem[SA12]{SA12}{A.S. Sadullaev and B.I. Abdullaev, \it Potential theory in the class of
m-subharmonic functions, \rm Trudy Matematicheskogo Instituta imeni V.A. Steklova, vol. 279 (2012),  166-192.}

\bibitem[T63]{T63} {A.F. Timan, \it Theory of approximation of functions of real variable, \rm (New York, 1963). Translated from Russian.}

\bibitem[W12]{W12} {Y. Wang, \it A Viscosity Approach to the Dirichlet Problem for Complex Monge-Amp\`ere Equations, \rm Math. Z. 272 (2012), no. 1-2, 497-513.}

\bibitem[YL10]{YL10}{H. Yong and X. Lu, \it Regularity of radial solutions to the complex Hessian equations, \rm J. Part. Diff. Eq., 23 (2010), pp. 147-157.}



\end{thebibliography}
\end{document}